\documentclass[11pt,reqno]{amsart}
\usepackage[a4paper,margin=1in]{geometry}
\usepackage[T1]{fontenc}
\usepackage{setspace}
\usepackage{amsmath,amssymb,amsthm,mathtools}

\usepackage{aliascnt}
\usepackage{tikz}
\usetikzlibrary{arrows.meta}
\usepackage{microtype}
\usepackage{hyperref}
\usepackage[nameinlink,noabbrev]{cleveref}

\newtheorem{theorem}{Theorem}[section]
\newaliascnt{lemma}{theorem}
\newtheorem{lemma}[lemma]{Lemma}
\aliascntresetthe{lemma}
\newaliascnt{proposition}{theorem}

\aliascntresetthe{proposition}
\newaliascnt{corollary}{theorem}
\newtheorem{corollary}[corollary]{Corollary}
\aliascntresetthe{corollary}

\theoremstyle{definition}
\newaliascnt{remark}{theorem}

\aliascntresetthe{remark}
\newaliascnt{definition}{theorem}
\newtheorem{definition}[definition]{Definition}
\aliascntresetthe{definition}

\crefname{theorem}{Theorem}{Theorems}
\Crefname{theorem}{Theorem}{Theorems}
\crefname{lemma}{Lemma}{Lemmas}
\Crefname{lemma}{Lemma}{Lemmas}
\crefname{proposition}{Proposition}{Propositions}
\Crefname{proposition}{Proposition}{Propositions}
\crefname{corollary}{Corollary}{Corollaries}
\Crefname{corollary}{Corollary}{Corollaries}
\crefname{remark}{Remark}{Remarks}
\Crefname{remark}{Remark}{Remarks}
\crefname{definition}{Definition}{Definitions}
\Crefname{definition}{Definition}{Definitions}

\newcommand{\Aa}{A_{\alpha}}
\newcommand{\ra}{\rho_{\alpha}}
\newcommand{\Ls}{\Lambda_s}

\DeclareMathOperator{\spec}{Spec}
\DeclareMathOperator{\diag}{diag}

\newcommand{\tp}{\mathrm{T}}

\begin{document}

\title{Sharp $A_\alpha$-spectral conditions for odd $[1,b]$-factors when $\alpha>1/2$}
\author{Silin Huang}
\date{}

\begin{abstract}
We solve, for all sufficiently large even orders,
the problem proposed by Chen~et~al.
on sharp $A_\alpha$-spectral conditions
for the existence of odd $[1,b]$-factors when $\alpha>1/2$.
Chen~et~al. showed that
every connected graph of even order $n$
with no odd $[1,b]$-factor
has $A_\alpha$-spectral radius at most
$\max_{1\le s\le k}\rho_\alpha(G_s)$,
where
$G_s=K_s\nabla\left(K_{n-(b+1)s-1}\cup(bs+1)K_1\right)$
and
$k=\lfloor(n-2)/(b+1)\rfloor$.
Thus the problem reduces
to finding the graph with the largest $A_\alpha$-spectral radius
among these obstruction graphs.
We prove that, for every $\alpha\in(1/2,1)$,
$\max_{1\le s\le k}\rho_\alpha(G_s)=\max\{\rho_\alpha(G_1),\rho_\alpha(G_k)\}$.
Moreover, for each fixed odd $b\ge 3$ and every even
$n\ge N_b=(b+1)\max\{2b+3,14\}+2$,
there exists a unique
$\alpha=\alpha_\ast(n,b)\in(1/2,1)$ at which
$\rho_\alpha(G_1)=\rho_\alpha(G_k)$.
Consequently, $G_1$ is the unique extremal graph for
$1/2<\alpha<\alpha_\ast(n,b)$,
both $G_1$ and $G_k$ are extremal at
$\alpha=\alpha_\ast(n,b)$,
and $G_k$ is the unique extremal graph for
$\alpha_\ast(n,b)<\alpha<1$.
This gives the exact $A_\alpha$-spectral threshold,
together with the sharp exceptional graphs,
for odd $[1,b]$-factors when $\alpha>1/2$ and $n\ge N_b$.
\end{abstract}

\maketitle

\section{Introduction}
For a graph $G$, let $A(G)$ and $D(G)$
be its adjacency matrix
and its diagonal degree matrix, respectively.
Nikiforov~\cite{Nikiforov2017}
introduced the $\Aa$-matrix, defined by
\begin{equation}
    \Aa(G)=\alpha D(G)+(1-\alpha)A(G),\qquad \alpha\in[0,1].
\end{equation}
This matrix interpolates between the adjacency matrix,
one half of the signless Laplacian, and the degree matrix.
We write $\ra(G)$ for the spectral radius of $\Aa(G)$.

Let $b\ge 1$ be odd.
An \emph{odd $[1,b]$-factor} of $G$ is a spanning subgraph $H$
such that the degree of every vertex $v\in V(H)$ in $H$,
denoted by $d_H(v)$, is odd, and
\begin{equation}
    1\le d_H(v)\le b\qquad \text{for all }v\in V(G).
\end{equation}
Amahashi~\cite{Amahashi1985} proved that
$G$ contains an odd $[1,b]$-factor if and only if
\begin{equation}
    o(G-S)\le b|S|\qquad \text{for every }S\subseteq V(G),
\end{equation}
where $o(G-S)$ denotes the number of odd components
(i.e. components that have an odd number of vertices) of $G-S$.
This criterion is the starting point
of the spectral approach to odd factors.

The interplay between spectral radius
and the existence of graph factors
has attracted considerable attention in recent years.
For instance, for the adjacency spectral radius,
O~\cite{O2021} obtained sharp conditions for perfect matchings,
and sharp conditions for odd $[1,b]$-factors
and general $[a,b]$-factors were subsequently obtained
in \cite{FanLinLu2022,ZhouLiu2023,WeiZhang2023}.
For regular graphs,
Kim~et~al.~\cite{KimOParkRee2020} and O~\cite{O2022}
gave eigenvalue conditions for odd $[1,b]$-factors
and general $[a,b]$-factors respectively.
For the $\Aa$-spectral radius,
Zhao~et~al.~\cite{ZhaoHuangWang2021}
studied perfect matchings, and
Chen~et~al.~\cite{ChenWenHa2024}
solved the odd $[1,b]$-factor problem for $\alpha\in[0,\frac12]$.
For other related $\Aa$-spectral factor results,
such as conditions for path-factors, (fractional) $[a,b]$-factors,
extendability, and some other factor problems,
readers may refer to
\cite{ZhouZhangSun2024,ZhengWangHuang2024,HaWen2025,LvLiXu2025,FanLiuAo2025}.

In this paper, we study the following problem
proposed by Chen~et~al.~\cite{ChenWenHa2024}.
\begin{quote}
    \textbf{Problem~1.5}.
    Let $\alpha\in(\frac12,1)$
    and let $G$ be a connected graph of even order $n$.
    Investigate the lower bound on $\ra(G)$
    to guarantee the existence of an odd $[1,b]$-factor.
\end{quote}

In~\cite[Section~3]{ChenWenHa2024},
the problem is reduced to the graph family
\begin{equation}\label{eq:introGs}
    G_s=K_s\nabla\left(K_{n-(b+1)s-1}\cup(bs+1)K_1\right),
    \qquad 1\le s\le k\coloneqq\left\lfloor\frac{n-2}{b+1}\right\rfloor.
\end{equation}
In other words, if $G$ is a connected graph of even order $n$
with no odd $[1,b]$-factors, then
\begin{equation}
    \ra(G)\le\max_{1\le s\le k}\ra(G_s).
\end{equation}
Hence Problem~1.5 is reduced to the comparison of
\begin{equation}
    \Lambda_s(\alpha;n,b)\coloneqq\rho_\alpha(G_s),
    \qquad 1\le s\le k.
\end{equation}
For $\alpha\le\frac12$,~\cite{ChenWenHa2024}
shows that $G_1$ is extremal
for all $n>b+2+\alpha+\frac{2(b+1)(b+2-\alpha)^2}{b}$.
The case $\alpha>\frac12$ is the part left open there
and is solved in this paper
for $n\ge N_b\coloneqq(b+1)\max\{2b+3,14\}+2$.

Our first main result shows that,
once the problem is reduced to the family $\{G_s\}_{s=1}^k$,
it is enough to compare the two graphs $G_1$ and $G_k$.

\begin{theorem}\label{thm:intro-endpoint-reduction}
    Let $n$ be even and let $b\ge 3$ be odd. Set
    $k=\left\lfloor\frac{n-2}{b+1}\right\rfloor$.
    Assume $k\ge 1$. Then for every $\alpha\in(\frac12,1)$,
    \begin{equation}
        \max_{1\le s\le k}\ra(G_s)=\max\{\ra(G_1),\ra(G_k)\}.
    \end{equation}
\end{theorem}

Our second main result determines the unique value of $\alpha$
at which $\ra(G_1)=\ra(G_k)$,
for all even $n$ above an explicit threshold.

\begin{theorem}\label{thm:intro-endpoint-crossing}
    Let $n$ be even and let $b\ge 3$ be odd.
    Then, for every even $n\ge N_b$,
    there exists a unique
    $\alpha_\ast(n,b)\in(\frac12,1)$ such that
    \begin{equation}
        \ra(G_1)=\ra(G_k).
    \end{equation}
    More precisely, we have
    \begin{equation}
        \begin{aligned}
            \alpha_\ast(n,b)
            &=1-\frac{(b+1)^2}{bn}+\frac{(b+1)^2\left(b^3+b\ell-b-1\right)}{b^3n^2}+O(n^{-3}), \\
            \ell&=n-(b+1)k-1.
        \end{aligned}
    \end{equation}
    Consequently, for $\frac12<\alpha<\alpha_\ast(n,b)$,
    $G_1$ is the unique extremal graph
    in $\{G_s\}_{s=1}^k$;
    for $\alpha=\alpha_\ast(n,b)$,
    both $G_1$ and $G_k$ are the extremal graphs
    in $\{G_s\}_{s=1}^k$;
    for $\alpha_\ast(n,b)<\alpha<1$,
    $G_k$ is the unique extremal graph
    in $\{G_s\}_{s=1}^k$.
\end{theorem}

Combining the two theorems above yields the following theorem.

\begin{theorem}\label{thm:intro-general}
    Let $G$ be a connected graph of even order
    $n\ge N_b=(b+1)\max\{2b+3,14\}+2$.
    Let $b\ge 3$ be odd and let $\alpha_\ast(n,b)$
    be the parameter determined from
    \Cref{thm:intro-endpoint-crossing}.
    Then:
    \begin{enumerate}
        \item If $\frac12<\alpha<\alpha_\ast(n,b)$ and
        \begin{equation}
            \ra(G)\ge\ra(G_1),
        \end{equation}
        then $G$ contains an odd $[1,b]$-factor
        unless $G\cong G_1$.

        \item If $\alpha_\ast(n,b)<\alpha<1$ and
        \begin{equation}
            \ra(G)\ge\ra(G_k),
        \end{equation}
        then $G$ contains an odd $[1,b]$-factor
        unless $G\cong G_k$.

        \item If $\alpha=\alpha_\ast(n,b)$ and
        \begin{equation}
            \ra(G)\ge\ra(G_1)=\ra(G_k),
        \end{equation}
        then $G$ contains an odd $[1,b]$-factor
        unless $G\cong G_1$ or $G\cong G_k$.
    \end{enumerate}
\end{theorem}

The structure of this paper is as follows.
\Cref{sec:reduction,sec:quotient}
collect the reduction and quotient-matrix tools.
\Cref{sec:compareGs} proves that only
$G_1$ and $G_k$ need to be compared, and
\Cref{sec:endpoint} compares
$G_1$ and $G_k$ and determines the value
of $\alpha$ for which they have the same spectral radius.

\section{\texorpdfstring{Preliminaries and the graph family $G_s$}{Preliminaries and the graph family G\_s}}\label{sec:reduction}
All graphs in this paper
are simple, finite, undirected, and connected,
unless stated otherwise.
We write $\nabla$ for the join of two graphs
and $\cup$ for the disjoint union of two graphs.
Throughout the paper, we suppose that $n$ is even and $b$ is odd.
This is because if $b$ is even, then an odd $[1,b]$-factor
is equivalent to an odd $[1,b-1]$-factor.
Moreover, if $b=1$, this is exactly the perfect-matching case
and has been treated by Zhao et al.~\cite{ZhaoHuangWang2021}.
Hence, in this paper, we always assume that $b\ge 3$.

We begin with Amahashi's criterion.

\begin{lemma}[Amahashi~{\cite[Theorem~2]{Amahashi1985}}]\label{lem:amahashi}
    Let $G$ be a graph and let $b$ be a positive odd integer.
    Then $G$ contains an odd $[1,b]$-factor if and only if
    \begin{equation}
        o(G-S)\le b|S|\qquad \text{for every }S\subseteq V(G).
    \end{equation}
\end{lemma}

We first verify that each graph in this family we stated
is indeed an obstruction.

\begin{lemma}\label{lem:Gs-obstruction}
    For every $1\le s\le k$, the graph
    \begin{equation}
        G_s=K_s\nabla\left(K_{n-(b+1)s-1}\cup(bs+1)K_1\right)
    \end{equation}
    contains no odd $[1,b]$-factor.
\end{lemma}
\begin{proof}
    Let $S=V(K_s)$. Then
    \begin{equation}
        G_s-S=K_{r_s}\cup t_sK_1,
        \qquad \text{where }r_s\coloneqq n-(b+1)s-1,
        \,t_s\coloneqq bs+1.
    \end{equation}
    Since $n$ is even and $b+1$ is even, $r_s$ is odd.
    Hence $G_s-S$ has exactly $t_s+1=bs+2$ odd components.
    Therefore
    \begin{equation}
        o(G_s-S)=bs+2>bs=b|S|.
    \end{equation}
    By \Cref{lem:amahashi}, $G_s$ has no odd $[1,b]$-factor.
\end{proof}

The next two lemmas are part of the argument of Chen~et~al.
One can verify that
none of these steps depends on $\alpha\le\frac12$,
so the same proof works for all $\alpha\in[0,1)$.

\begin{lemma}\label{lem:cwh-reduction}
    Let $\alpha\in[0,1)$
    and let $G$ be a connected graph of even order $n$
    with no odd $[1,b]$-factor. Set
    \begin{equation}
        k\coloneqq\left\lfloor\frac{n-2}{b+1}\right\rfloor,
        \qquad
        G_s=K_s\nabla\left(K_{n-(b+1)s-1}\cup(bs+1)K_1\right)
        \quad (1\le s\le k).
    \end{equation}
    Then there exists $s\in\{1,\dots,k\}$ such that
    \begin{equation}
        \ra(G)\le\ra(G_s).
    \end{equation}
    Equivalently,
    \begin{equation}
        \ra(G)\le\max_{1\le s\le k}\ra(G_s).
    \end{equation}
\end{lemma}
\begin{proof}
    This is proved in
    \cite[Claims~3.5 and~3.6]{ChenWenHa2024}.
\end{proof}

The next lemma proves the sharpness
of the obstruction graph family $\{G_i\}_{i=1}^k$.

\begin{lemma}\label{lem:sharp-transfer}
    Let $\alpha\in[0,1)$
    and let $G$ be a connected graph of even order $n$
    with no odd $[1,b]$-factor.
    Let $k$ and $G_s$ be as in \Cref{lem:cwh-reduction}.
    If
    \begin{equation}
        \ra(G)\ge\max_{1\le i\le k}\ra(G_i),
    \end{equation}
    then there exists $s\in\{1,\dots,k\}$ such that $G\cong G_s$.
\end{lemma}
\begin{proof}
    This condition forces equality in \Cref{lem:cwh-reduction},
    and the equality clauses in \cite[Claims~3.4--3.6]{ChenWenHa2024}
    therefore give $G\cong G_s$
    for some $s$ with $\ra(G_s)=\max_{1\le i\le k}\ra(G_i)$.
\end{proof}

For later convenience, we fix the following notation.

\begin{definition}\label{def:notation}
    Define
    \begin{equation}
        k\coloneqq\left\lfloor\frac{n-2}{b+1}\right\rfloor,
        \qquad
        r_s\coloneqq n-(b+1)s-1,
        \qquad
        t_s\coloneqq bs+1,
        \qquad
        \Ls(\alpha)\coloneqq\ra(G_s),
    \end{equation}
    where $G_s$ is given by \eqref{eq:introGs}.
    With this notation, $G_s=K_s\nabla(K_{r_s}\cup t_sK_1)$.
    For later convenience, we also write
    \begin{equation}
        \ell\coloneqq r_k=n-(b+1)k-1.
    \end{equation}
    It holds that $1\le \ell\le b$ and $\ell$ is odd.
\end{definition}

\section{Quotient matrices and estimates for the spectral radii}\label{sec:quotient}
This section reduces $\ra(G_s)$
to the Perron root of a $3\times 3$ quotient matrix.
Moreover, to compare $\ra(G_1)$ and $\ra(G_k)$,
we obtain two expansions of $\Lambda_s$: one for fixed $s$,
and one for the case $\epsilon n\le s\le k$.

We start by introducing the concept of a \emph{quotient matrix}.

\begin{definition}
    Let $M=(m_{uv})$ be an $n\times n$ matrix
    whose rows and columns are indexed by a finite set $X$, and let
    $\pi=\{X_1,X_2,\dots,X_t\}$ be a partition of $X$.
    We say that $\pi$ is an \emph{equitable partition} of $M$
    if, for every $i,j\in\{1,2,\dots,t\}$, the quantity
    \begin{equation}
        b_{ij}=\sum_{v\in X_j}m_{uv}
    \end{equation}
    is independent of the choice of $u\in X_i$.
    In this case, the $t\times t$ matrix $M/\pi\coloneqq(b_{ij})$
    is called the \emph{quotient matrix} of $M$
    with respect to $\pi$.
\end{definition}

An important property of quotient matrices is that
if the original matrix is nonnegative and irreducible,
then the quotient matrix has the same spectral radius
as the original matrix.

\begin{lemma}[{\cite[Theorem~2.3 and 2.5]{YouYangSoXi2019}}]\label{lem:quotient-spectrum}
    Let $M/\pi$ be the quotient matrix of $M$
    with respect to $\pi$. Then
    \begin{equation}
        \spec(M/\pi)\subseteq\spec(M).
    \end{equation}
    Moreover, if $M$ is nonnegative and irreducible, then
    \begin{equation}
        \rho(M/\pi)=\rho(M).
    \end{equation}
\end{lemma}

We show that an equitable partition for $A(G)$
is also equitable for $\Aa(G)$.

\begin{lemma}\label{lem:equitable-Aalpha}
    Let $G$ be a graph, $A(G)$ be its adjacency matrix, and
    let $\pi=\{V_1,\dots,V_t\}$
    be an equitable partition of $A(G)$.
    Let $A(G)/\pi\eqqcolon B(G)=(b_{ij})$
    be the quotient matrix of $A(G)$ with respect to $\pi$.
    Then $\pi$ is also equitable for $\Aa(G)$,
    and the corresponding quotient matrix is
    \begin{equation}
        \Aa(G)/\pi=\alpha\Delta+(1-\alpha)B(G)\eqqcolon B_{\alpha}(G),
    \end{equation}
    where $\Delta=\diag(d_1,\dots,d_t)$, and
    \begin{equation}
        d_i=\sum_{j=1}^t b_{ij}
    \end{equation}
    are the row sums of $B$.
    Moreover, if $G$ is connected and $\alpha<1$, then
    \begin{equation}
        \rho(B_{\alpha}(G))=\rho(\Aa(G)).
    \end{equation}
\end{lemma}
\begin{proof}
    Let $u\in V_i$.
    Since $\pi$ is equitable for $A(G)$,
    the number of neighbors of $u$ in $V_j$ is $b_{ij}$,
    which is independent of the choice of $u\in V_i$.
    Hence every vertex in $V_i$ has the same degree
    \begin{equation}
        d_i=\sum_{j=1}^t b_{ij}.
    \end{equation}
    Therefore, for $u\in V_i$,
    the sum of the entries of the row of $\Aa(G)$
    indexed by $u$ over the block $V_j$ is
    \begin{equation}
        \sum_{v\in V_j}(\Aa(G))_{uv}=
        \begin{cases}
            \alpha d_i+(1-\alpha)b_{ii}, & j=i,\\
            (1-\alpha)b_{ij}, & j\ne i.
        \end{cases}
    \end{equation}
    This depends only on $i$ and $j$, so $\pi$ is also equitable
    for $\Aa(G)$. The corresponding quotient matrix is
    \begin{equation}
        \Aa(G)/\pi=\alpha\Delta+(1-\alpha)B(G)=B_{\alpha}(G).
    \end{equation}
    If $G$ is connected and $\alpha<1$, then $\Aa(G)$ is nonnegative
    and irreducible, so \Cref{lem:quotient-spectrum} gives
    \begin{equation}
        \rho(B_{\alpha}(G))=\rho(\Aa(G)).
    \end{equation}
    This completes the proof.
\end{proof}

To obtain an equitable partition for $A(G_s)$,
we fix $1\le s\le k$ and partition
$V(G_s)=V\left(K_s\nabla\left(K_{r_s}\cup t_sK_1\right)\right)$
into three parts:
\begin{equation}
    V_1=V(K_s),\qquad V_2=V(K_{r_s}),\qquad V_3=V(t_sK_1).
\end{equation}
One can verify that
the partition $\{V_1,V_2,V_3\}$ is equitable for $A(G_s)$,
hence also equitable for $\Aa(G_s)$ by \Cref{lem:equitable-Aalpha}.
For later convenience, write
\begin{equation}\label{eq:delta-uvw}
    \beta\coloneqq 1-\alpha,
    \qquad u_s\coloneqq n-1-\beta(n-s),
    \qquad v_s\coloneqq n-bs-2-\beta s,
    \qquad w_s\coloneqq s-\beta s.
\end{equation}
A direct computation shows that the corresponding quotient matrix is
\begin{equation}
    B_s\coloneqq B_\alpha(G_s)=\begin{bmatrix}
        u_s & \beta r_s & \beta t_s\\
        \beta s & v_s & 0\\
        \beta s & 0 & w_s
    \end{bmatrix}.
\end{equation}
By setting $P_s=\diag(\sqrt{s},\sqrt{r_s},\sqrt{t_s})$,
$B_s$ is similar and thus cospectral to the real symmetric matrix
\begin{equation}
    \widetilde{B}_s=P_s B_s P_s^{-1}=
    \begin{bmatrix}
    u_s & \beta\sqrt{s r_s} & \beta\sqrt{s t_s}\\
    \beta\sqrt{s r_s} & v_s & 0\\
    \beta\sqrt{s t_s} & 0 & w_s
    \end{bmatrix}.
\end{equation}
We then have
\begin{equation}
    \Ls(\alpha)\coloneqq\ra(G_s)=\rho(B_s)=\rho(\widetilde{B}_s).
\end{equation}

We now derive the characteristic equation
satisfied by $\rho(\widetilde{B}_s)$.

\begin{lemma}\label{lem:schur}
    For $1\le s\le k$,
    $\Ls=\Ls(\alpha)$ satisfies
    \begin{equation}\label{eq:schur-poly}
        (\Ls-u_s)(\Ls-v_s)(\Ls-w_s)-\beta^2 s\left[r_s(\Ls-w_s)+t_s(\Ls-v_s)\right]=0.
    \end{equation}
    Moreover, we have
    \begin{equation}\label{eq:schur}
        \Ls=u_s+\beta^2 s\left(\frac{r_s}{\Ls-v_s}+\frac{t_s}{\Ls-w_s}\right).
    \end{equation}
\end{lemma}
\begin{proof}
    \eqref{eq:schur-poly}
    is simply the characteristic equation $\det(\Ls I-\widetilde{B}_s)=0$.
    Moreover, since $\widetilde{B}_s$ is nonnegative and irreducible,
    by the Perron--Frobenius theorem,
    its Perron root, which is exactly its spectral radius $\Ls$,
    satisfies $\Ls>\max\{u_s,v_s,w_s\}$.
    Therefore we may divide
    \eqref{eq:schur-poly} by $(\Ls-v_s)(\Ls-w_s)$
    and get \eqref{eq:schur}.
\end{proof}

This gives an upper bound for $\Ls$
when $u_s>\max\{v_s,w_s\}$,
which will be useful later.

\begin{lemma}\label{lem:Lambda-upper-bound}
    Write
    \begin{equation}
        \Lambda_s^{+}\coloneqq
        u_s+\beta^2s\left(\frac{r_s}{u_s-v_s}
        +\frac{t_s}{u_s-w_s}\right).
    \end{equation}
    Then if $u_s-v_s>0$ and $u_s-w_s>0$,
    it holds that $\Ls\le\Lambda_s^{+}$.
\end{lemma}
\begin{proof}
    Define the function
    \begin{equation}
        f_s(x)\coloneqq u_s+\beta^2 s\left(\frac{r_s}{x-v_s}+\frac{t_s}{x-w_s}\right),
        \qquad x>\max\{v_s,w_s\}.
    \end{equation}
    By \Cref{lem:schur}, $\Ls=f_s(\Ls)$.
    Moreover,
    \begin{equation}
        f_s'(x)=-\beta^2 s\left(\frac{r_s}{(x-v_s)^2}+\frac{t_s}{(x-w_s)^2}\right)<0,
    \end{equation}
    so $f_s$ is strictly decreasing on its domain.

    If $u_s-v_s>0$ and $u_s-w_s>0$,
    then $u_s>\max\{v_s,w_s\}$,
    so $u_s$ lies in the domain of $f_s$.
    By \Cref{lem:schur}, we have $\Ls>u_s$.
    Therefore the monotonicity of $f_s$ yields
    \begin{equation}
        \Ls=f_s(\Ls)\le f_s(u_s)=\Lambda_s^{+}.
    \end{equation}
    This completes the proof.
\end{proof}

We next record two estimates of $\Lambda_s$:
one for fixed $s$, and one for the range $\epsilon n\le s\le k$.

\begin{theorem}\label{thm:schur-asymptotics}
    Fix an odd integer $b\ge 3$ and let $n\to\infty$
    through even integers. Set $c=n(1-\alpha)$.
    Then the following estimates hold.
    \begin{enumerate}
        \item Fix a positive integer $s$,
        and let $I$ be a fixed compact interval contained in
        $(bs+1,+\infty)$. Then, uniformly for $c\in I$,
        \begin{equation}
            \Ls\left(1-\frac{c}{n}\right)
            =n-bs-2-\frac{cs}{n}+\frac{c^2s}{(c-bs-1)n}+O(n^{-2}).
        \end{equation}
        The constant in the $O(n^{-2})$ term may depend
        on $b$, $s$, and $I$, but is independent of $c$ and $n$.
        \item Fix $\epsilon\in(0,\frac{1}{b+1})$.
        If $\epsilon n\le s\le k$,
        then, uniformly for $c$ in any fixed bounded interval
        $J\subset(0,+\infty)$,
        \begin{equation}\label{eq:theta-expansion}
            \Ls\left(1-\frac{c}{n}\right)
            =n-1-c+\frac{cs}{n}
            +\frac{c^2}{n}h_b\left(\frac{s}{n}\right)+O(n^{-2}),
        \end{equation}
        where
        \begin{equation}
            h_b(\theta)\coloneqq\frac{1-(b+1)\theta}{b}
            +\frac{b\theta^2}{1-\theta}.
        \end{equation}
        The constant in the $O(n^{-2})$ term may depend
        on $b$, $\epsilon$, and $J$,
        but is independent of $c$, $s$, and $n$.
    \end{enumerate}
\end{theorem}
\begin{proof}
    Recall from \eqref{eq:schur} that, after setting $\beta=c/n$,
    \begin{equation}\label{eq:schur-c}
        \Ls=u_s+\frac{c^2s}{n^2}
        \left(\frac{r_s}{\Ls-v_s}+\frac{t_s}{\Ls-w_s}\right).
    \end{equation}

    We first prove the first estimate. Set
    \begin{equation}\label{eq:delta-s}
        \Delta_s\coloneqq c-bs-1>0,
    \end{equation}
    and write $\Ls=v_s+r$ with $r>0$. Since
    $u_s-v_s=-(c-bs-1)+O(n^{-1})$, \eqref{eq:schur-c} gives
    \begin{equation}\label{eq:schur-r-c-2}
        r^2+\left(\Delta_s-\frac{2cs}{n}\right)r
        =\frac{c^2sr_s}{n^2}
        +\frac{c^2st_s\,r}{n^2(v_s+r-w_s)}.
    \end{equation}
    Here
    \begin{equation}
        r_s=n+O(1),\qquad t_s=O(1),\qquad v_s-w_s=n+O(1).
    \end{equation}
    Since $I$ is compact and disjoint from $bs+1$,
    there exists $\eta>0$ such that $\Delta_s\ge\eta$ on $I$.
    Thus, for all large $n$,
    \begin{equation}
        r^2+\frac{\eta}{4}r\le \frac{C}{n},
    \end{equation}
    and hence $r=O(n^{-1})$. Returning to \eqref{eq:schur-r-c-2},
    we get
    \begin{equation}
        \Delta_s r=\frac{c^2s}{n}+O(n^{-2}),
    \end{equation}
    and therefore
    \begin{equation}
        r=\frac{c^2s}{(c-bs-1)n}+O(n^{-2}).
    \end{equation}
    Since $v_s=n-bs-2-cs/n$, the first estimate follows.

    We now prove the second estimate.
    Let $c$ range in a fixed bounded
    interval $J\subset(0,+\infty)$ and put
    \begin{equation}
        \theta_s\coloneqq\frac{s}{n}
        \in\left[\epsilon,\frac{1}{b+1}\right].
    \end{equation}
    Write $\Ls=u_s+r$ with $r>0$, and define
    \begin{equation}
        \Gamma_s\coloneqq u_s-v_s=b\theta_sn-c+1+2c\theta_s,
        \qquad
        \Omega_s\coloneqq u_s-w_s=(1-\theta_s)n-c-1+2c\theta_s.
    \end{equation}
    Both $\Gamma_s$ and $\Omega_s$ are bounded
    from above and below
    by positive constant multiples of $n$,
    uniformly for $c\in J$ and
    $\epsilon n\le s\le k$.
    Hence \eqref{eq:schur-c} gives $r=O(n^{-1})$.
    Moreover,
    \begin{equation}
        \frac{r_s}{\Gamma_s+r}=\frac{r_s}{\Gamma_s}+O(n^{-2}),
        \qquad
        \frac{t_s}{\Omega_s+r}=\frac{t_s}{\Omega_s}+O(n^{-2}),
    \end{equation}
    and therefore
    \begin{equation}\label{eq:r-theta-reduced}
        r=\frac{c^2s}{n^2}
        \left(\frac{r_s}{\Gamma_s}+\frac{t_s}{\Omega_s}\right)
        +O(n^{-3}).
    \end{equation}
    Since
    \begin{equation}
        \frac{r_s}{\Gamma_s}
        =\frac{1-(b+1)\theta_s}{b\theta_s}+O(n^{-1}),
        \qquad
        \frac{t_s}{\Omega_s}
        =\frac{b\theta_s}{1-\theta_s}+O(n^{-1}),
    \end{equation}
    and $s=\theta_sn$, \eqref{eq:r-theta-reduced} yields
    \begin{equation}
        r=\frac{c^2}{n}h_b(\theta_s)+O(n^{-2}).
    \end{equation}
    Together with $u_s=n-1-c+cs/n$, this proves \eqref{eq:theta-expansion}.
\end{proof}

Specializing these estimates to the two endpoint graphs gives
the expansions used later to locate the crossing point.

\begin{corollary}\label{cor:Lambda-1-expansion}
    Fix $c_0>b+1$. Then, uniformly for $c\in[c_0,2b+1]$,
    \begin{equation}\label{eq:end-l1}
        \Lambda_1\left(1-\frac{c}{n}\right)
        =n-1-(b+1)+\frac{c(b+1)}{(c-b-1)n}+O(n^{-2}).
    \end{equation}
\end{corollary}
\begin{proof}
    Apply the fixed-$s$ part of \Cref{thm:schur-asymptotics} with $s=1$.
\end{proof}

\begin{corollary}\label{cor:Lambda-k-expansion}
    Fix $c_0>b+1$. Then, uniformly for $c\in[c_0,2b+1]$,
    \begin{equation}\label{eq:end-lk}
        \Lambda_k\left(1-\frac{c}{n}\right)
        =n-1-\frac{b}{b+1}c
        +\frac{c(c-\ell-1)}{(b+1)n}+O(n^{-2}),
    \end{equation}
    where $\ell=n-(b+1)k-1$.
\end{corollary}
\begin{proof}
    Put
    \begin{equation}\label{eq:theta-k}
        \theta_k\coloneqq\frac{k}{n}
        =\frac{1}{b+1}-\frac{\ell+1}{(b+1)n}.
    \end{equation}
    Applying the second part of \Cref{thm:schur-asymptotics}
    with $s=k$ gives
    \begin{equation}
        \Lambda_k\left(1-\frac{c}{n}\right)
        =n-1-c+\frac{ck}{n}
        +\frac{c^2}{n}h_b(\theta_k)+O(n^{-2}).
    \end{equation}
    Since $h_b(1/(b+1))=1/(b+1)$ and
    $\theta_k=1/(b+1)+O(n^{-1})$, the result follows.
\end{proof}

\section{\texorpdfstring{It is enough to compare the spectral radii of $G_1$ and $G_k$}{It is enough to compare the spectral radii of G\_1 and G\_k}}\label{sec:compareGs}
In this section, for fixed $n$, $b$, and $\alpha$,
we compare the spectral radii of graphs in $\{G_s\}_{s=1}^k$.

The characteristic polynomial of $\widetilde{B}_s$ is
\begin{equation}
    \chi_s(x)\coloneqq\det(xI-\widetilde{B}_s)
    =(x-u_s)(x-v_s)(x-w_s)
    -\beta^2s\left[r_s(x-w_s)+t_s(x-v_s)\right].
\end{equation}
On the interval $x>\max\{v_s,w_s\}$, we have
\begin{equation}
    \frac{\chi_s(x)}{(x-v_s)(x-w_s)}
    =x-u_s-\beta^2s\left(\frac{r_s}{x-v_s}
    +\frac{t_s}{x-w_s}\right).
\end{equation}
The derivative of the right-hand side with respect to $x$ is
\begin{equation}
    1+\beta^2s\left(\frac{r_s}{(x-v_s)^2}
    +\frac{t_s}{(x-w_s)^2}\right)>0.
\end{equation}
Since $\Ls>\max\{v_s,w_s\}$ and $\chi_s(\Ls)=0$ by \Cref{lem:schur},
this quotient is strictly increasing and reaches $0$ at $x=\Ls$.
Therefore, for every $x>\max\{v_s,w_s\}$,
\begin{equation}\label{eq:chi-sign-criterion}
    \chi_s(x)\ge0\quad\Longleftrightarrow\quad x\ge\Ls,
    \qquad
    \chi_s(x)>0\quad\Longleftrightarrow\quad x>\Ls .
\end{equation}

\begin{lemma}\label{lem:chi-concave}
    Let $\alpha\in(\frac12,1)$,
    let $n$ be even, and let $b\ge 3$ be odd. Set
    $k=\left\lfloor\frac{n-2}{b+1}\right\rfloor$.
    Then, for fixed $x\ge u_k$,
    the function $s\mapsto\chi_s(x)$
    is strictly concave on $[1,k]$.
\end{lemma}
\begin{proof}
    Write
    \begin{equation}
        n=(b+1)k+\ell+1,\qquad 1\le \ell\le b.
    \end{equation}
    Treating $s$ as a real variable, a direct calculation gives
    \begin{equation}\label{eq:chi-second-derivative}
        \frac{\partial^2 \chi_s(x)}{\partial s^2}
        =-2\left(b(x-u_k)+\beta H_s\right),
    \end{equation}
    where
    \begin{equation}\label{eq:Hs-definition}
        H_s\coloneqq 3b\left((b+2)\beta-1\right)s
        +b(n+k-1)-2b\beta n+3b\beta+2\beta-1 .
    \end{equation}
    Since $x\ge u_k$ and $\beta>0$, it remains to prove that
    $H_s>0$ for $1\le s\le k$.

    The coefficient of $s$ in $H_s$ is
    $3b((b+2)\beta-1)$. If $0<\beta\le 1/(b+2)$, then
    $H_s\ge H_k$. Moreover $H_k$ is affine in $\beta$, and
    \begin{equation}
        H_k\big|_{\beta=0}
        =b\left((b-1)k+\ell\right)-1>0,
    \end{equation}
    while
    \begin{equation}
        H_k\big|_{\beta=\frac{1}{b+2}}
        =\frac{b\left(b^2k+2bk+b\ell+2k\right)}{b+2}>0.
    \end{equation}
    Thus $H_s>0$ in this case.

    If $1/(b+2)\le\beta<1/2$, then $H_s\ge H_1$. Again
    $H_1$ is affine in $\beta$, and
    \begin{equation}
        H_1\big|_{\beta=\frac{1}{b+2}}
        =\frac{b\left(b^2k+2bk+b\ell+2k\right)}{b+2}>0,
    \end{equation}
    while
    \begin{equation}
        H_1\big|_{\beta=\frac12}
        =\frac{b(3b+2k+1)}{2}>0.
    \end{equation}
    Hence $H_s>0$ also in this case. Therefore
    $\partial^2\chi_s(x)/\partial s^2<0$ on $[1,k]$, so
    $s\mapsto\chi_s(x)$ is strictly concave.
\end{proof}

The next theorem shows that the spectral radius of $G_s$
is dominated by that of one of the two endpoints, $G_1$ and $G_k$.

\begin{theorem}\label{thm:endpoint-reduction}
    Let $n$ be even and let $b\ge 3$ be odd. Set
    $k=\left\lfloor\frac{n-2}{b+1}\right\rfloor$.
    Assume $k\ge 1$. Then for every $\alpha\in(\frac12,1)$,
    \begin{equation}
        \max_{1\le s\le k}\ra(G_s)=\max\{\ra(G_1),\ra(G_k)\}.
    \end{equation}
\end{theorem}
\begin{proof}
    If $k\le 2$, the assertion is immediate. Hence assume $k\ge 3$.
    Set
    \begin{equation}
        \Lambda_\ast\coloneqq\max\{\Lambda_1,\Lambda_k\}.
    \end{equation}
    It suffices to prove that
    \begin{equation}
        \Lambda_\ast>\Ls\qquad (2\le s\le k-1).
    \end{equation}

    We first check that $\Lambda_\ast>\max\{v_s,w_s\}$ for every $s$.
    By the Perron--Frobenius argument in \Cref{lem:schur},
    $\Lambda_i>\max\{u_i,v_i,w_i\}$ for $i=1,k$.
    Since $v_s=n-bs-2-\beta s$ is decreasing in $s$, we have
    \begin{equation}
        v_s\le v_1<\Lambda_1\le \Lambda_\ast
        \qquad (1\le s\le k).
    \end{equation}
    Also, $w_s=\alpha s$ is increasing in $s$, and
    \begin{equation}
        u_k-w_k
        =\alpha n-1-(2\alpha-1)k
        \ge \alpha n-1-(2\alpha-1)\frac n2
        =\frac n2-1>0.
    \end{equation}
    Hence
    \begin{equation}
        w_s\le w_k<u_k<\Lambda_k\le\Lambda_\ast
        \qquad (1\le s\le k).
    \end{equation}
    Thus
    \begin{equation}\label{eq:xstar-domain}
        \Lambda_\ast>\max\{v_s,w_s\}
        \qquad (1\le s\le k).
    \end{equation}

    Moreover, by \eqref{eq:chi-sign-criterion}
    and the definition of $\Lambda_\ast$,
    \begin{equation}
        \chi_1(\Lambda_\ast)\ge0,
        \qquad
        \chi_k(\Lambda_\ast)\ge0.
    \end{equation}
    Since $\Lambda_\ast>u_k$, \Cref{lem:chi-concave} implies that
    $s\mapsto\chi_s(\Lambda_\ast)$ is strictly concave on $[1,k]$. Therefore,
    for every $2\le s\le k-1$,
    \begin{equation}
        \chi_s(\Lambda_\ast)>
        \frac{k-s}{k-1}\chi_1(\Lambda_\ast)
        +\frac{s-1}{k-1}\chi_k(\Lambda_\ast)
        \ge0.
    \end{equation}
    Applying \eqref{eq:chi-sign-criterion} once more gives
    \begin{equation}
        \Lambda_\ast>\Ls
        \qquad (2\le s\le k-1).
    \end{equation}
    This completes the proof.
\end{proof}

\section{\texorpdfstring{The comparison of the spectral radii of $G_1$ and $G_k$}{The comparison of the spectral radii of G\_1 and G\_k}}\label{sec:endpoint}
We now compare the spectral radius
of the two endpoint graphs $G_1$ and $G_k$.
By the last section, this comparison suffices because
the graph with the largest spectral radius in the family
must be one of these two endpoint graphs.
Throughout this section, write
\begin{equation}
    n=(b+1)k+\ell+1\quad(1\le\ell=r_k\le b),
    \qquad c=n(1-\alpha).
\end{equation}
Thus $0<c<n/2$ for $\alpha\in(\frac12,1)$. Put
\begin{equation}
    D_n(c)\coloneqq
    \Lambda_1\left(1-\frac cn\right)-\Lambda_k\left(1-\frac cn\right).
\end{equation}
We shall prove that, for each fixed $n$ under consideration,
$D_n(c)$ has exactly one zero.
Specifically, we will prove that
$D_n(c)$ is negative on a left interval,
strictly increasing through a unique zero on a middle interval,
and positive on a right interval.
\Cref{fig:Dn-sign-pattern} illustrates this behavior.

\begin{figure}[htb]
    \centering
    \begin{tikzpicture}[
        x=0.5cm,
        y=0.5cm,
        >=Stealth,
        every node/.style={font=\small}
    ]
        \draw[->] (0,0) -- (13,0) node[right] {$c$};
        \draw[->] (0,-1.2) -- (0,6.2) node[above] {$D_n(c)$};
        \draw[densely dashed] (4,-1) -- (4,5)
            node[above] {$c=b+1$};
        \draw[densely dashed] (7,-1) -- (7,6)
            node[above] {$c=2b+1$};
        \draw[thick,blue]
            plot[smooth] coordinates {
                (0.400,-0.090)
                (1.200,-0.267)
                (2.000,-0.436)
                (3.000,-0.596)
                (4.000,-0.530)
                (4.700,-0.235)
                (5.115,0.000)
                (5.300,0.112)
                (6.000,0.558)
                (7.000,1.219)
                (10.000,3.202)
                (11.000,3.847)
                (12.000,4.482)
            };
        \fill[blue] (5.115,0) circle (2pt);
        \node[below] at (5.115,0) {$c_\ast$};
        \node[below left] at (3,-0.596) {$D_n<0$};
        \node[above left] at (11,3.847) {$D_n>0$};
    \end{tikzpicture}
    \caption{The sign pattern of $D_n(c)$ in the endpoint comparison.
    The plotted points are computed from the quotient matrices
    for $b=3$ and $n=N_3=58$.}
    \label{fig:Dn-sign-pattern}
\end{figure}
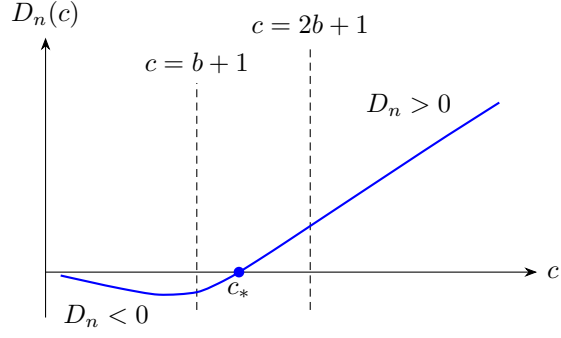

For the rest of the section, set
\begin{equation}
    K_b\coloneqq \max\{2b+3,14\},
    \qquad
    N_b\coloneqq (b+1)K_b+2.
\end{equation}
Equivalently,
\begin{equation}
    K_b=\begin{cases}
        14, & b\in\{3,5\}, \\
        2b+3, & b\ge 7,\,b\text{ odd},
    \end{cases}
    \qquad
    N_b=\begin{cases}
        58, & b=3, \\
        86, & b=5, \\
        2b^2+5b+5, & b\ge7,\,b\text{ odd}.
    \end{cases}
\end{equation}
Note that $n\ge N_b$ is equivalent to $k\ge K_b$.

\begin{lemma}\label{lem:endpoint-left}
    Let $b\ge3$ be odd and let $n\ge N_b$ be even.
    Then
    \begin{equation}
        D_n(c)<0\qquad (0<c\le b+1).
    \end{equation}
\end{lemma}
\begin{proof}
    If $0<c\le b+\frac12$,
    a direct calculation gives
    $u_1-v_1\ge\frac12$ and $u_1-w_1\ge\frac n2$.
    Writing $\Lambda_1=u_1+r$ in the Schur equation then gives
    \begin{equation}\label{eq:step1-1-1}
        r=
        \frac{c^2}{n^2}
        \left(
            \frac{r_1}{\Lambda_1-v_1}
            +\frac{b+1}{\Lambda_1-w_1}
        \right)
        \le\frac{2c^2}{n}.
    \end{equation}
    Here the last inequality follows from
    \begin{equation}
        \frac{r_1}{\Lambda_1-v_1}+\frac{b+1}{\Lambda_1-w_1}
        \le
        \frac{n-b-2}{u_1-v_1}+\frac{b+1}{u_1-w_1}
        \le 2(n-b-2)+2(b+1)
        =2(n-1)<2n.
    \end{equation}
    Since $\Lambda_k>u_k$ and
    \begin{equation}\label{eq:step1-1-2}
        u_k-u_1=\frac{c(k-1)}{n},
    \end{equation}
    while $k\ge K_b$, we have
    \begin{equation}\label{eq:step1-1-3}
        \frac{k-1}{n}>\frac{2(b+\frac12)}{n}.
    \end{equation}

    Combining these facts, we have
    \begin{equation}
        u_k-u_1
        \overset{\eqref{eq:step1-1-2}}{=}
        \frac{c(k-1)}{n}
        \overset{\eqref{eq:step1-1-3}}{>}
        \frac{2c\left(b+\frac12\right)}{n}
        \ge
        \frac{2c^2}{n}
        \overset{\eqref{eq:step1-1-1}}{\ge}
        r,
    \end{equation}
    Hence
    \begin{equation}
        \Lambda_1=u_1+r<u_1+(u_k-u_1)=u_k<\Lambda_k,
    \end{equation}
    therefore $D_n(c)<0$ for $0<c\le b+\frac12$.

    It remains to consider $b+\frac12\le c\le b+1$.
    Let $a=c\sqrt{r_1}/n$, $d=c\sqrt{b+1}/n$, and put $x=u_1+3/4$.

    Since $r_1=n-(b+1)-1=n-b-2<n$, we have
    \begin{equation}\label{eq:step1-2-1}
        a^2=\frac{c^2r_1}{n^2}
        \le
        \frac{(b+1)^2}{n}.
    \end{equation}
    We also have
    \begin{equation}
        d^2=\frac{c^2(b+1)}{n^2}
        \le
        \frac{(b+1)^3}{n^2}.
    \end{equation}
    Moreover, since $x-w_1=u_1-w_1+3/4$ and
    \begin{equation}
        u_1-w_1=n-2-c+\frac{2c}{n}\ge n-b-3\ge\frac n2,
    \end{equation}
    we obtain
    \begin{equation}\label{eq:step1-2-2}
        \frac{d^2}{x-w_1}
        \le
        \frac{(b+1)^3/n^2}{n/2}
        =
        \frac{2(b+1)^3}{n^3}.
    \end{equation}
    Finally, using $x-v_1,x-w_1>x-u_1=3/4$,
    it holds that
    \begin{equation}\label{eq:step1-2-3}
        \frac{\chi_1(x)}{(x-v_1)(x-w_1)}
        =x-u_1-\frac{a^2}{x-v_1}-\frac{d^2}{x-w_1}
        \ge\frac34-\frac{4(b+1)^2}{3n}
        -\frac{2(b+1)^3}{n^3}.
    \end{equation}
    
    We claim that \eqref{eq:step1-2-3} is positive
    for every $n\ge N_b$.
    Indeed, for $b=3\text{ (resp. }5\text{)}$,
    this is checked directly
    from $n\ge58\text{ (resp. }86\text{)}$.
    For $b\ge7$, the lower bound is at least
    \begin{equation}
        \frac34-\frac{4(b+1)^2}{3(2b^2+5b+5)}
        -\frac{2(b+1)^3}{(2b^2+5b+5)^3}>0,
    \end{equation}
    since, after multiplying both sides by
    $12(2b^2+5b+5)^3$, the numerator is
    \begin{equation}
        8b^6+92b^5+466b^4+1241b^3+1933b^2+1703b+701,
    \end{equation}
    whose coefficients are all positive.
    Therefore $x=u_1+3/4>\Lambda_1$.

    On the other hand,
    \begin{equation}
        u_k-u_1
        =\frac{c(k-1)}{n}
        \overset{n\le(b+1)(k+1)}{\ge}
        \frac{(b+\frac12)(K_b-1)}{(b+1)(K_b+1)}
        >\frac34.
    \end{equation}
    The last inequality is immediate for $b=3,5$,
    where $K_b=14$, and for $b\ge7$ it becomes
    \begin{equation}
        \frac{(b+\frac12)(2b+2)}{(b+1)(2b+4)}
        =\frac{2b+1}{2b+4}>\frac34.
    \end{equation}
    Thus $\Lambda_k>u_k>u_1+3/4=x>\Lambda_1$
    and therefore $D_n(c)<0$, completing the proof.
\end{proof}

\begin{lemma}\label{lem:endpoint-middle}
    Let $b\ge3$ be odd and let $n\ge N_b$ be even.
    Then $D_n$ is strictly increasing on $[b+1,2b+1]$.
\end{lemma}
\begin{proof}
    Consider the quotient matrices $\widetilde{B}_s(c)$.
    For a unit Perron vector
    $z_s=(p_s,q_s,\zeta_s)^\tp$ of $\widetilde{B}_s(c)$,
    differentiating the matrix $\widetilde{B}_s(c)$ with respect to $c$ gives
    \begin{equation}\label{eq:step2-1-1}
        \Lambda_s'(c)
        =z_s^\tp\widetilde{B}_s'(c)z_s
        =\left(-1+\frac{s}{n}\right)p_s^2
        -\frac{s}{n}q_s^2
        -\frac{s}{n}\zeta_s^2
        +\frac{2\sqrt{s r_s}}{n}p_sq_s
        +\frac{2\sqrt{s t_s}}{n}p_s\zeta_s.
    \end{equation}
    We claim that
    \begin{equation}
        \Lambda_1'(c)>-\frac23,
        \qquad
        \Lambda_k'(c)<-\frac23
        \qquad (b+1\le c\le2b+1).
    \end{equation}

    First consider $s=1$. Let
    $a=c\sqrt{r_1}/n$, $d=c\sqrt{b+1}/n$, and
    $\gamma=\sqrt{3/2}$.  Put $x=v_1+\gamma a$.
    We claim that
    \begin{equation}
        x-u_1-\frac{a^2}{x-v_1}-\frac{d^2}{x-w_1}
        =v_1-u_1+\left(\gamma-\frac1\gamma\right)a
        -\frac{d^2}{x-w_1}>0
    \end{equation}
    for all $n\ge N_b$.
    
    From
    \begin{align}
        v_1-u_1&=c-b-1-\frac{2c}{n}\ge-\frac{2c}{n}, \\
        \left(\gamma-\frac1\gamma\right)a&=\frac1{\sqrt6}\cdot\frac{c\sqrt{r_1}}{n}
        \ge\frac1{\sqrt6}\cdot\frac{c\sqrt{n/2}}{n}=\frac{c}{\sqrt{12n}}, \\
        \frac{d^2}{x-w_1}&=\frac{c^2(b+1)/n^2}{x-w_1}
        \le\frac{(2b+1)^2(b+1)/n^2}{n/2}=\frac{2(2b+1)^2(b+1)}{n^3},
    \end{align}
    we have
    \begin{equation}
        v_1-u_1+\left(\gamma-\frac1\gamma\right)a
        -\frac{d^2}{x-w_1}
        \ge c\left(\frac1{\sqrt{12n}}-\frac2n\right)
        -\frac{2(2b+1)^2(b+1)}{n^3}.
    \end{equation}
    We prove that the right-hand side of the above is positive.
    Since $n\ge58$ and $c\ge b+1$, we have
    \begin{equation}
        \begin{aligned}
            c\left(\frac1{\sqrt{12n}}-\frac2n\right)
            -\frac{2(2b+1)^2(b+1)}{n^3}
            &\ge\frac{b+1}{40\sqrt n}
            -\frac{2(2b+1)^2(b+1)}{n^3} \\
            &=(b+1)\left(\frac1{40\sqrt n}-\frac{2(2b+1)^2}{n^3}\right)>0.
        \end{aligned}
    \end{equation}
    The last inequality follows from
    $n^{5/2}>80(2b+1)^2$,
    which is immediate from $n\ge N_b$.
    Hence $\Lambda_1\le v_1+\gamma a$.
    Therefore the Perron equation gives
    \begin{equation}\label{eq:step2-1-2}
        \begin{bmatrix}
            u_1 & a & d\\
            a & v_1 & 0\\
            d & 0 & w_1
        \end{bmatrix}
        \begin{bmatrix}
            p_1 \\ q_1 \\ \zeta_1
        \end{bmatrix}
        =\Lambda_1
        \begin{bmatrix}
            p_1 \\ q_1 \\ \zeta_1
        \end{bmatrix}
        \implies
        a p_1+v_1q_1=\Lambda_1 q_1
        \implies
        \frac{q_1}{p_1}=\frac{a}{\Lambda_1-v_1}\ge\frac1\gamma,
    \end{equation}
    and therefore $p_1^2\le\gamma^2/(1+\gamma^2)=3/5$
    since $p_1^2+q_1^2\le 1$.
    Dropping the positive off-diagonal terms
    of~\eqref{eq:step2-1-1} gives
    \begin{equation}
        \Lambda_1'(c)
        \ge\left(-1+\frac1n\right)p_1^2-\frac1n(1-p_1^2)
        =\left(-1+\frac2n\right)p_1^2-\frac1n
        >-\frac23.
    \end{equation}

    Now consider $s=k$.
    Let $A=c\sqrt{k\ell}/n$ and $E=c\sqrt{k(bk+1)}/n$.
    Since $\Lambda_k>u_k$, as in~\eqref{eq:step2-1-2},
    the Perron equations give
    \begin{equation}
        \frac{q_k}{p_k}\le \frac{A}{u_k-v_k},
        \qquad
        \frac{\zeta_k}{p_k}\le \frac{E}{u_k-w_k}.
    \end{equation}
    Moreover, on $b+1\le c\le2b+1$,
    \begin{equation}
        u_k-v_k\ge bk-2b,\qquad
        u_k-w_k\ge bk-2b .
    \end{equation}
    Substituting these into~\eqref{eq:step2-1-1}
    and dropping the negative terms yields
    \begin{equation}
        \Lambda_k'(c)
        \le
        -1+\frac{k}{n}
        +
        \frac{2c k(\ell+bk+1)}{n^2(bk-2b)} .
    \end{equation}
    Since $k\ge K_b$, $\ell\le b$, $c\le2b+1$, and
    $n\ge(b+1)k$, the last positive term
    \begin{equation}
        \frac{2(2b+1)(b+bk+1)}
        {(b+1)^2kb(k-2)}
        \le \frac1{12}.
    \end{equation}
    For $b=3,5$ this is verified by substituting $k\ge K_b=14$.
    For $b\ge7$, this is decreasing in $k$,
    so it is enough to check $k=2b+3$.
    After clearing denominators this is equivalent to
    \begin{equation}
        4b^5+16b^4-73b^3-226b^2-141b-24>0,
    \end{equation}
    which holds
    since the maximum real root of this polynomial is approximately $4.16$.
    Therefore
    \begin{equation}
        \Lambda_k'(c)
        \le-\frac{b}{b+1}+\frac1{12}
        \le-\frac23.
    \end{equation}
    Thus $D_n'(c)=\Lambda_1'(c)-\Lambda_k'(c)>0$
    throughout the interval, completing the proof.
\end{proof}

\begin{lemma}\label{lem:endpoint-right}
    Let $b\ge3$ be odd and let $n\ge N_b$ be even.
    Then
    \begin{equation}
        D_n(c)>0\qquad (2b+1\le c<n/2).
    \end{equation}
\end{lemma}
\begin{proof}
    For $c\ge 2b+1$, $v_1>\max\{u_1,w_1\}$. Thus
    \begin{equation}
        \Lambda_1>v_1=n-b-2-\frac cn.
    \end{equation}
    On the other hand, by \Cref{lem:Lambda-upper-bound},
    \begin{equation}
        \Lambda_k
        \le u_k+\frac{c^2k}{n^2}
        \left(\frac{\ell}{u_k-v_k}+\frac{bk+1}{u_k-w_k}\right).
    \end{equation}
    This is valid because $c<n/2$ and $n>2k$ gives
    \begin{align}
        u_k-v_k
        &=bk+1-c\left(1-\frac{2k}{n}\right)
          >\frac{(b+1)k+1-\ell}{2}>0, \\
        u_k-w_k
        &=n-k-1-c\left(1-\frac{2k}{n}\right)
          >\frac{(b+1)k+\ell-1}{2}>0.
    \end{align}
    Consequently, since $1\le\ell\le b<k+1$, it holds that
    \begin{align}
        u_k-v_k&>\frac{(b+1)k+1-\ell}{2}
        \ge\frac{(b+1)k+1-b}{2}
        =\frac{bk+(k+1-b)}{2}
        >\frac{bk}{2}, \\
        u_k-w_k&>\frac{(b+1)k}{2}.
    \end{align}
    Therefore
    \begin{equation}
        \frac{k}{n^2}\left(\frac{\ell}{u_k-v_k}+\frac{bk+1}{u_k-w_k}\right)
        <
        \frac1{n^2}\left(\frac{2\ell}{b}+\frac{2(bk+1)}{b+1}\right)
        \le
        \frac1{n^2}
        \left(2+\frac{2(bk+1)}{b+1}\right)
        \le
        \frac{3b}{(b+1)^2n},
    \end{equation}
    where the last inequality follows from $n\ge (b+1)k+2$
    and $b(b+1)k\ge2b^2+4$, the latter being immediate from
    $b\ge3$ and $k\ge3$. Hence
    \begin{equation}
        D_n(c)=\Lambda_1-\Lambda_k
        >\left(n-b-2-\frac cn\right)-\left(u_k+\frac{3bc^2}{(b+1)^2n}\right)
        =-b-1+c\left(1-\frac{k+1}{n}\right)-\frac{3bc^2}{(b+1)^2n}.
    \end{equation}

    Since
    \begin{equation}
        1-\frac{k+1}{n}
        =\frac{b}{b+1}-\frac{b-\ell}{(b+1)n}
        \ge\frac{b}{b+1}-\frac{b}{(b+1)n},
    \end{equation}
    we get
    \begin{equation}
        D_n(c)
        \ge-b-1+c\left(\frac{b}{b+1}-\frac{b}{(b+1)n}\right)
        -\frac{3bc^2}{(b+1)^2n}
        \eqqcolon Q_n(c).
    \end{equation}
    Notice that $Q_n(c)$
    is a concave quadratic function of $c$,
    so on $2b+1\le c\le n/2$ its minimum is attained at an endpoint.
    We claim that at $c=2b+1$ it is
    \begin{equation}
        Q_n(2b+1)=\frac{b^2-b-1}{b+1}
        -\frac{b(2b+1)}{(b+1)n}
        -\frac{3b(2b+1)^2}{(b+1)^2n}>0,
    \end{equation}
    and at $c=n/2$ it equals
    \begin{equation}
        Q_n(n/2)=\frac{bn(2b-1)}{4(b+1)^2}
        -\frac{b}{2(b+1)}-b-1>0.
    \end{equation}
    Indeed, the first of these inequalities follows from
    \begin{equation}
        \frac{b^2-b-1}{b+1}
        >\frac{b(2b+1)}{(b+1)n}+\frac{3b(2b+1)^2}{(b+1)^2n},
    \end{equation}
    which is immediate for $b=3,5$ after substituting
    $N_b=58,86$. For $b\ge 7$, after substituting
    $N_b=2b^2+5b+5$ and multiplying both sides by
    $(b+1)^2(2b^2+5b+5)$, this is equivalent to
    \begin{equation}
        2b^5+5b^4-13b^3-27b^2-19b-5,
    \end{equation}
    which is positive for all $b\ge 7$
    since the maximum real root of this polynomial is approximately $2.60$.
    The second follows from
    \begin{equation}
        \frac{bn(2b-1)}{4(b+1)^2}
        >b+1+\frac{b}{2(b+1)},
    \end{equation}
    again using $n\ge N_b$.
    This is immediate for $b=3,5$ after substituting
    $N_b=58,86$. For $b\ge7$, after substituting
    $N_b=2b^2+5b+5$ and multiplying both sides by
    $4(b+1)^2$, this is equivalent to
    \begin{equation}
        4b^4+4b^3-9b^2-19b-4>0,
    \end{equation}
    which is positive for all $b\ge 7$
    since the maximum real root of this polynomial is approximately $1.82$.
    Thus $D_n(c)>0$ on the asserted interval,
    completing the proof.
\end{proof}

Now that the three lemmas are settled,
we can derive the main result of this section.

\begin{theorem}\label{thm:endpoint-crossing}
    Fix an odd integer $b\ge3$ and set
    \begin{equation}
        N_b=(b+1)\max\{2b+3,14\}+2 .
    \end{equation}
    For every
    even $n\ge N_b$, there exists a unique
    $\alpha_\ast(n,b)\in(\frac12,1)$ such that
    \begin{equation}
        \Lambda_1(\alpha_\ast)=\Lambda_k(\alpha_\ast).
    \end{equation}
    Consequently,
    \begin{align}
        \Lambda_1(\alpha)>\Lambda_k(\alpha)
        &\quad\Longleftrightarrow\quad
        \frac12<\alpha<\alpha_\ast(n,b), \\
        \Lambda_k(\alpha)>\Lambda_1(\alpha)
        &\quad\Longleftrightarrow\quad
        \alpha_\ast(n,b)<\alpha<1, \\
        \Lambda_1(\alpha)=\Lambda_k(\alpha)
        &\quad\Longleftrightarrow\quad
        \alpha=\alpha_\ast(n,b).
    \end{align}
    Moreover,
    \begin{equation}
        \alpha_\ast(n,b)
        =1-\frac{(b+1)^2}{bn}
        +\frac{(b+1)^2(b^3+b\ell-b-1)}{b^3n^2}
        +O(n^{-3}).
    \end{equation}
\end{theorem}
\begin{proof}
    By \Cref{lem:endpoint-left}, $D_n(b+1)<0$,
    while by \Cref{lem:endpoint-right}, $D_n(2b+1)>0$.
    Since $D_n$ is strictly increasing on $[b+1,2b+1]$ by
    \Cref{lem:endpoint-middle},
    it has exactly one zero $c_\ast\in(b+1,2b+1)$.
    The same two endpoint lemmas exclude all zeros
    outside this interval.
    Hence $\alpha_\ast=1-c_\ast/n\in(\frac12,1)$
    is the unique crossing point.

    It remains to locate the zero.
    The endpoint expansions
    \Cref{cor:Lambda-1-expansion,cor:Lambda-k-expansion}
    give, uniformly on compact subintervals of $(b+1,2b+1)$,
    \begin{equation}\label{eq:Dn-FG-expansion}
        D_n(c)=F(c)+\frac1nG(c)+O(n^{-2}),
    \end{equation}
    where
    \begin{equation}
        F(c)\coloneqq\frac{bc}{b+1}-(b+1),
        \qquad
        G(c)\coloneqq
        \frac{c(b+1)}{c-b-1}-\frac{c(c-\ell-1)}{b+1}.
    \end{equation}
    The leading term $F(c)$ has the unique zero
    \begin{equation}
        c_0\coloneqq\frac{(b+1)^2}{b}.
    \end{equation}
    Since $b+1<c_0<2b+1$, we may fix small $\delta>0$ such that
    $b+1<c_0-\delta<c_0+\delta<2b+1$.
    From \eqref{eq:Dn-FG-expansion}, or just its leading part,
    \begin{equation}
        D_n(c)=F(c)+O(n^{-1})
    \end{equation}
    uniformly on $[c_0-\delta,c_0+\delta]$.
    Hence $D_n(c_0-\delta)<0<D_n(c_0+\delta)$
    for all sufficiently large $n$.
    Since the zero is unique, $c_\ast\in(c_0-\delta,c_0+\delta)$.

    We now refine this location by one Taylor step.
    Write
    \begin{equation}
        c_\ast=c_0+\frac{a}{n}+O(n^{-2}).
    \end{equation}
    Substituting this into \eqref{eq:Dn-FG-expansion} and using
    $F(c_0)=0$ gives
    \begin{equation}
        0=D_n(c_\ast)
        =\frac1n\left(aF'(c_0)+G(c_0)\right)+O(n^{-2}).
    \end{equation}
    Therefore
    \begin{equation}
        a=-\frac{G(c_0)}{F'(c_0)}
        =-\frac{(b+1)^2(b^3+b\ell-b-1)}{b^3}.
    \end{equation}
    Thus
    \begin{equation}
        c_\ast
        =\frac{(b+1)^2}{b}
        -\frac{(b+1)^2(b^3+b\ell-b-1)}{b^3n}
        +O(n^{-2}).
    \end{equation}
    The stated expansion for $\alpha_\ast=1-c_\ast/n$ follows.
    Finally $c=n(1-\alpha)$ is strictly decreasing in $\alpha$,
    so the two sign alternatives are exactly as claimed.
\end{proof}

\section{Proofs of the main theorems in the Introduction}\label{sec:main-theorems}
We are now ready to prove the theorems stated in the Introduction,
which completes the paper.

\begin{proof}[Proof of \Cref{thm:intro-endpoint-reduction}]
    This is exactly \Cref{thm:endpoint-reduction}.
\end{proof}

\begin{proof}[Proof of \Cref{thm:intro-endpoint-crossing}]
    The existence, uniqueness, sign alternatives,
    and the expression for the crossing of $\Lambda_1$ and $\Lambda_k$
    are exactly \Cref{thm:endpoint-crossing}.
    Combining these with the endpoint reduction
    in \Cref{thm:endpoint-reduction}
    gives the asserted extremal graphs in the whole family
    $\{G_s\}_{s=1}^k$.
\end{proof}

\begin{proof}[Proof of \Cref{thm:intro-general}]
    Suppose that $G$ has no odd $[1,b]$-factor.  By
    \Cref{lem:sharp-transfer}, if
    \begin{equation}
        \ra(G)\ge\max_{1\le s\le k}\ra(G_s),
    \end{equation}
    then $G\cong G_s$ for some graph $G_s$
    that is extremal in the family.
    The extremal graphs are identified in
    \Cref{thm:intro-endpoint-crossing}:
    for $\frac12<\alpha<\alpha_\ast(n,b)$ it is $G\cong G_1$;
    for $\alpha_\ast(n,b)<\alpha<1$ it is $G\cong G_k$;
    and for $\alpha=\alpha_\ast(n,b)$ the extremal graphs are
    $G\cong G_1$ and $G\cong G_k$.
    This is precisely the conclusion of \Cref{thm:intro-general}.
\end{proof}

\bibliographystyle{amsplain-with-doi}
\bibliography{references}

\end{document}